\documentclass{amsart}

\usepackage{amsmath,amsxtra,amsthm,amssymb,stmaryrd,verbatim}
\usepackage{tikz-cd}
\usepackage{tcolorbox}
\usepackage{hyperref}

\hypersetup{colorlinks=true, linkcolor=blue}

\newtheorem{Thm}{Theorem}
\newtheorem{Prop}[Thm]{Proposition}

\newtheorem{Cor}[Thm]{Corollary}

\DeclareMathOperator{\Ker}{Ker}

\DeclareMathOperator{\End}{End}

\DeclareMathOperator{\rad}{rad}

\DeclareMathOperator{\bM}{\mathbb M}

\begin{document}

\title{On Lenagan's Theorem for finite length bimodules}
\author{Andrew Hubery}
\address{Bielefeld University\\33501 Bielefeld\\Germany}
\email{hubery@math.uni-bielefeld.de}
\begin{abstract}
We offer a self-contained proof of Lenagan's Theorem which does not rely on Goldie's Theorem.
\end{abstract}
\subjclass[2010]{Primary 16P20; Secondary 16D20}

\maketitle

\section{Introduction}

Lenagan's Theorem states that for a bimodule ${}_\Gamma M_\Lambda$, if ${}_\Gamma M$ has finite length and $M_\Lambda$ is Noetherian, then $M_\Lambda$ also has finite length. This theorem first appeared in \cite{Lenagan}, and as Lam writes
\begin{quotation}
Indeed, although the argument above is quite short, it seemed to have used the full force of Goldie's First Theorem, and it is not clear at all how one could have proved [it] otherwise. \cite[p. 333]{Lam}
\end{quotation}
One important consequence is that the left and right Artin radicals of a Noetherian ring agree. We can also regard Lenagan's Theorem as a generalisation of the classical result that a left Artinian ring is right Artinian if and only if it is right Noetherian. This latter is a direct consequence of the Hopkins-Levitzki Theorem.

In fact, Lenagan's Theorem was generalised by Crawley-Boevey \cite{Crawley-Boevey}:  if ${}_\Gamma M$ is Artinian and $M_\Lambda$ is Noetherian, then ${}_\Gamma M$ and $M_\Lambda$ both have finite length. We give a new proof of this result without recourse to Goldie's Theorem. Instead it follows from a strengthening of the Hopkins--Levitzki Theorem. In a similar way we also obtain a result of Bj\"ork on subrings of semiprimary rings. The starting point for our proof is a result of Camps and Dicks on semilocal rings.

Let $\Lambda$ be a ring, with Jacobson radical $\rad(\Lambda)$. We call $\Lambda$ semilocal if $\bar\Lambda:=\Lambda/\rad(\Lambda)$ is semisimple, and semiprimary if moreover $\rad(\Lambda)$ is nilpotent. It is well-known that the endomorphism ring of a finite length object in an abelian category is always semiprimary.

\section{Finite length modules}

We begin with a beautiful proof, due to Camps and Dicks \cite{Camps-Dicks}. We follow the proof in the second edition \cite{Camps-Dicks2}, which incorporates ideas from Camillo and Nielsen \cite{Camillo-Nielsen}.

\begin{Thm}
Let $M$ be an Artinian object in some abelian category, and set $E:=\End(M)$. If $\Lambda\subset E$ is a subring such that $\Lambda^\times=\Lambda\cap E^\times$, then $\Lambda$ is semilocal.
\end{Thm}

\begin{proof}
Observe first that for all $x,y\in E$ we have
\[ \Ker(x-xyx) = \Ker(x) \oplus \Ker(1-xy), \]
induced by the idempotent endomorphism $xy$ on $\Ker(x-xyx)$. Also, $M$ satisfies the ascending chain condition with respect to direct summands: given $U_i\leq M$ with $U_i=U_{i-1}\oplus V_i$, the $V_i$ are eventually all zero. For, given such a sequence, we can construct the descending chain of submodules $\bigoplus_{i>n}V_i$.

Now let $I\leq\Lambda$ be a maximal right ideal. Take $x\not\in I$ such that $\Ker(x)$ is maximal, with respect to direct summands, in $\{\Ker(y):y\not\in I\}$. We claim that $x\Lambda\cap I\subset\rad(\Lambda)$, so suppose $xy\in I$. Then for all $z\in\Lambda$ we have $\Ker(x-xyzx)=\Ker(x)\oplus\Ker(1-xyz)$. Since $x-xyz\not\in I$, we see that $1-xyz$ is injective. By Fitting's Lemma it is invertible in $E$, so is invertible in $\Lambda$ by assumption.

It follows that $\bar\Lambda$ equals $\bar I\oplus x\bar\Lambda$. In particular, $x\bar\Lambda$ is simple, so the socle of $\bar\Lambda$ is not contained in any maximal right ideal, and hence $\bar\Lambda$ is semisimple.
\end{proof}

The following proposition is a strengthening of the Hopkins--Levitzki Theorem, which deals with the case $\Lambda=E$.

\begin{Prop}
Let $\Lambda$ subring of semiprimary ring $E$, and $M$ a right $E$-module. Then $M_\Lambda$ is Noetherian if and only if it is Artinian.
\end{Prop}

\begin{proof}
Since $E$ is semiprimary, $M_E$ has finite Loewy length, and hence we may assume that $M_E$ is semisimple. Then $M_\Lambda$ Artinian or Noetherian implies the same for $M_E$, and hence we may assume that $M_E$ is simple. Note that this already proves the result when $\Lambda=E$; that is, $M_E$ is Noetherian if and only if it is Artinian.

After taking the quotient by the annihilator of $M$, we may assume that $E=\bM_n(\Delta)$ for a division ring $\Delta$, and hence regard $M$ as a $\Delta$-$E$-bimodule. We claim that, in this situation, $\Lambda$ is semiprimary, so $M_\Lambda$ has finite length as above.

Since $E_E\cong M_E^n$, we know that $E_\Lambda$, and hence also $\Lambda_\Lambda$, is Artinian or Noetherian. If $\Lambda$ is right Artinian, then it is semiprimary, and we are done. Assume therefore that $E_\Lambda$ is right Noetherian.

Suppose $x\in\Lambda$ has inverse $y\in E$. For some $d$ we have $y^d\in\sum_{i<d}y^i\Lambda$, and thus $y\in\sum_{i<d}x^{d-i-1}\Lambda\subset\Lambda$. Since ${}_\Delta M$ has finite length, it follows from the theorem that $\Lambda$ is semilocal. Also, $E\cdot\rad^n(\Lambda)$ form a descending chain of left $\Delta$-modules, so stabilises. Nakayama's Lemma then gives $E\cdot\rad^n(\Lambda)=0$ for some $n$. Thus $\rad^n(\Lambda)=0$, and $\Lambda$ is semiprimary as claimed.
\end{proof}

The next corollary appears as Theorem 3.11 in \cite{Bjork}.
\begin{Cor}
Let $E$ be semiprimary and $\Lambda\subset E$ a subring. If $E_\Lambda$ is Noetherian, then $\Lambda$ is right Artinian.
\end{Cor}

\begin{proof}
Apply the proposition to $M=E$.
\end{proof}

\begin{Cor}[Lenagan,Crawley-Boevey]
Let ${}_\Gamma M_\Lambda$ be a bimodule such that ${}_\Gamma M$ is Artinian and $M_\Lambda$ is Noetherian. Then ${}_\Gamma M$ and $M_\Lambda$ both have finite length.
\end{Cor}

\begin{proof}
It is enough to prove that $E:=\End_\Gamma(M)$ is semiprimary, since then $M_\Lambda$ has finite length by the proposition, so $\End_\Lambda(M)$ is also semiprimary, and hence ${}_\Gamma M$ has finite length by the proposition once more.

Now, the theorem tells us that that $E$ is semilocal. Also, the $\Gamma$-sub\-modules $M\cdot\rad^n(E)$ form a descending chain, so must stabilise. As $M_\Lambda$ is Noetherian, so too is $M_E$. Thus Nakayama's Lemma gives $M\cdot\rad^n(E)=0$ for some $n$, and hence that $\rad^n(E)=0$.
\end{proof}

\end{document}